\documentclass[a4paper, 12pt]{amsart}

\usepackage[T1]{fontenc}
\usepackage[ngerman,english]{babel}
\usepackage{amsmath}
\usepackage{amsthm}
\usepackage{amssymb}
\usepackage{amsfonts}
\usepackage[breaklinks]{hyperref}
\usepackage{tabularx}
\usepackage{multirow}

\usepackage{tikz}
\usetikzlibrary{matrix,arrows,patterns,shapes}
\usepackage{MnSymbol} 
\usepackage{nicefrac}
\usepackage{paralist} 

\theoremstyle{plain}
\newtheorem{theorem}{\bf{\textsc{Theorem}}}[section]

\newtheorem{theoremAlpha}{\bf{\textsc{Theorem}}}

\newtheorem{corollary}[theorem]{\bf{\textsc{Corollary}}}
\newtheorem{lemma}[theorem]{\bf{\textsc{Lemma}}}
\newtheorem{proposition}[theorem]{\bf{\textsc{Proposition}}}
\newtheorem*{conjecture}{\bf\textsc{Conjecture}}

\numberwithin{equation}{section}

\theoremstyle{definition}

\theoremstyle{definition}
\newtheorem{remark}[theorem]{\bf{\textsc{Remark}}}

\newcommand{\CC}{\mathbb{C}}				
\newcommand{\RR}{\mathbb{R}}				
\newcommand{\NN}{\mathbb{N}}				
\newcommand{\PP}{\mathbb{P}}				
\newcommand{\ZZ}{\mathbb{Z}}				

\newcommand{\Poly}{\mathcal P}			

\DeclareMathOperator{\Hom}{Hom}			
\DeclareMathOperator{\GL}{GL}				

\DeclareMathOperator{\Id}{Id}				
\DeclareMathOperator{\Ad}{Ad}				
\DeclareMathOperator{\ad}{ad}				

\renewcommand{\Re}{\operatorname{Re}}			
\renewcommand{\Im}{\operatorname{Im}}			
\DeclareMathOperator{\Det}{Det}			
\newcommand{\linebundle}{\mathcal L} 

\renewcommand{\b}[1]{{\bf #1}}
\newcommand{\Vp}{{\mf n^+}}
\newcommand{\Vm}{{\mf n^-}}

\newcommand{\Norm}[1]{\left\|#1\right\|}						
\newcommand{\at}[1]{\big\rvert_{#1}}								
\newcommand{\sP}[2]{\langle#1|#2\rangle}						
\newcommand{\JTP}[3]{\left\{#1,\,#2,\,#3\right\}}		
\newcommand{\B}[2]{B_{#1,\,#2}}											
\newcommand{\mf}[1]{{\mathfrak{#1}}} 								
\newcommand{\Set}[2]{\left\{#1\,\middle|\,#2\right\}} 
\newcommand{\set}[2]{\{#1\,|\,#2\}}								  

\newcommand{\trans}[1]{t_{#1}}											

\newcommand{\half}{{\nicefrac{1}{2}}}			

\renewcommand{\ss}{\textup{ss}}

\newcommand{\sC}{\mathcal C}
\newcommand{\sE}{\mathcal E}
\newcommand{\sN}{\mathcal N}
\newcommand{\sO}{\mathcal O}
\newcommand{\sU}{\mathcal U}

\renewcommand{\sP}{\mathcal P}

\newcommand{\del}{\partial}
\newcommand{\delbar}{\bar\partial}
\newcommand{\iCR}{{\bar D}}

\tikzset{inner sep=0pt, 
  root/.style={circle,draw,minimum size=5pt,thick}, 
  cross/.style={cross out,draw,minimum size=4pt,thick},
  doubleline/.style={double distance=1.5pt,thick},
}

\title[Existence of nearly holomorphic sections]
{Existence of nearly holomorphic sections\\on compact Hermitian symmetric spaces}
\author{Benjamin Schwarz} 


\address{Benjamin Schwarz, Universit\"{a}t Paderborn,
Fakult\"{a}t f\"{u}r Elektrotechnik, Informatik und Mathematik,
Institut f\"{u}r Mathematik, Warburger Str. 100,
33098 Paderborn, Germany}
\email{bschwarz@math.upb.de}

\begin{document}

\begin{abstract}
Let $X=U/K$ be a compact Hermitian symmetric space, and let $\sE$ be a $U$-homogeneous Hermitian vector bundle on $X$. In a previous paper, we showed that the space of nearly holomorphic sections is well-adapted for harmonic analysis in $L^2(X,\sE)$ provided that non-trivial nearly holomorphic sections do exist. Here we investigate the problem of extending local nearly holomorphic sections to global ones and prove the existence of non-trivial nearly holomorphic sections. This extends the results on the $U$-type decomposition of $L^2(X,\sE)$ from our previous paper.
\end{abstract}

\maketitle

\section*{Introduction}
In this paper we continue the investigation started in \cite{S12a} of nearly holomorphic sections in holomorphic vector bundles on compact Hermitian symmetric spaces. On any Kähler manifold $X$ and holomorphic vector bundle $\sE$ on $X$, a nearly holomorphic section $f\in\sN(X,\sE)$ is characterized by the condition that for any open set $\sU\subseteq X$ and Kähler potential $\Psi$ on $\sU$, $f$ is given by a polynomial expression
\begin{align}\label{eq:NHSlocally}
	f(z) = \sum_{\b i\in\NN^n} f_\b i(z)\,q(z)^\b i\quad (z\in\sU)
\end{align}
with holomorphic coefficients $f_\b i\in\sO(\sU,\sE_\sU)$, $f_\b i=0$ for almost all $\b i$, and $q_\ell(z):=\frac{\partial\Psi}{\partial z^\ell}$ for some chosen coordinates $(z^1,\ldots,z^n)$ on $\sU$. In \cite{S12a} we proved that the holomorphic coefficients $f_\b i$ are uniquely determined by the choice of the Kähler potential $\Psi$. Moreover, the identity theorem holds for nearly holomorphic sections, i.e., the restriction map 
\[
	\sN(X,\sE)\hookrightarrow\sN(\sU,\sE_\sU),\ f\mapsto f|_\sU
\]
is an embedding. Restricting to Hermitian symmetric spaces of compact type and homogeneous vector bundles, the main goal of this paper is to prove a characterization of the image of this embedding, and hence characterize those local nearly holomorphic sections that extend to global ones. As part of the proof we show the existence of non-trivial nearly holomorphic sections in \emph{any} homogeneous Hermitian vector bundle, which also extends our results from \cite{S12a} concerning harmonic analysis for sections in $\sE$.

Let us describe our results in more detail. Let $X=U/K$ be a Hermitian symmetric space of compact type, and let $\sE=U\times_KE$ be a $U$-homogeneous vector bundle. We assume that $\sE$ is irreducible in the sense that $E$ is an irreducible $K$-module. Let $\mf u$ and $\mf k$ denote the Lie algebras of $U$ and $K$. The complex structure of $X$ corresponds to a grading $\mf u_\CC=\mf n^+\oplus\mf k_\CC\oplus\mf n^-$, where $\mf n^\pm$ are abelian subalgebras which are irreducible under the adjoint action of $K$. We may identify $\mf n^+$ holomorphically with an open and dense subset of $X$ and choose a $K$-invariant Kähler potential on $\mf n^+$. In this setting, \eqref{eq:NHSlocally} yields an identification of $\sN(\mf n^+,\sE_{\mf n^+})$ with the space of $\sO(\mf n^+,E)$-valued polynomials on $(\mf n^+)^*\cong\mf n^-$. We prove that the coefficients $f_\b i$ corresponding to global nearly holomorphic sections are in fact polynomial, i.e., the restriction map yields an embedding into the space of $E$-valued polynomials on $\mf n^+\times\mf n^-$,
\[
	\iota_\sN:\sN(X,\sE)\hookrightarrow\Poly(\mf n^+\times\mf n^-,E),\ f\mapsto p_f
	\quad\text{with}\quad
	f|_{\mf n^+}(z) = p_f(z,q(z)).
\]
Here, $q=\partial\Psi$ is considered as a map from $\mf n^+$ to $\mf n^-\cong(\mf n^+)^*$. For the characterization of the image of $\iota_\sN$ we use the Hermitian structure on $\sE$, which is induced from the essentially unique $K$-invariant Hermitian inner product on $E$. Since $\mf n^+\subseteq X$ is dense, the $L^2$-norm of $f\in\sN(X,\sE)$ is given by a certain integral over $\mf n^+$. Therefore, elements of the image of $\iota_\sN$ necessarily satisfy the condition
\begin{align}\label{eq:L2conditionIntro}
	p(z,q(z))\in L^2(X,\sE).
\end{align}
Our main result states that this condition is also sufficient, i.e., $\iota_\sN$ maps $\sN(X,\sE)$ onto the space of polynomials $p\in\Poly(\mf n^+\times\mf n^-,E)$ satisfying \eqref{eq:L2conditionIntro}, denoted by $\Poly^2(\mf n^+\times\mf n^-,E)$. Proving estimates for the $L^2$-norm corresponding to particular polynomials, we also show that there always exist non-trivial nearly holomorphic sections. In summary,

\begin{theoremAlpha}[Theorem~\ref{thm:existenceofnearlyholsec}, Theorem~\ref{thm:MinorIntegral}]\ 
	\label{thm:A}\\
	Let $X=U/K$ be a compact Hermitian symmetric space, and $\sE=U\times_K E$ be an irreducible 
	$U$-homogeneous Hermitian vector bundle. The map 
	\[
		\iota_\sP: \Poly^2(\Vp\times\Vm,E)\to L^2(X,\sE),\ p\mapsto f_p(z):=p(z,q(z))
	\]
	is an isomorphism onto $\sN(X,\sE)$ with inverse $\iota_\sN$.	Moreover, $\sN(X,\sE)$ is 
	non-trivial.
\end{theoremAlpha}

As a consequence, it follows from our results in \cite{S12a} that $\sN(X,\sE)$ coincides with the space of $U$-finite vectors in $L^2(X,\sE)$, and hence contains all information about the $U$-type decomposition of $L^2(X,\sE)$. It turns out that there is a bijection between $U$-types in $L^2(X,\sE)$ and $K$-types in the subspace $\Poly^2(\mf n^-,E)$ of polynomials in $\Poly^2(\mf n^+\times\mf n^-,E)$ that are constant along $\mf n^+$. More explicitly, choose a Cartan subalgebra of $\mf k$, which then is also a Cartan subalgebra of $\mf u$. For an appropriate choice of positive root system, recall that if $V_\lambda$ is a $U$-type of highest weight $\lambda$ then the subspace $V_\lambda^{\mf n^+}$ of $\mf n^+$-invariants is a $K$-type of the same highest weight.

\begin{theoremAlpha}[Theorem~\ref{thm:UKequivalence}, Corollary~\ref{cor:multiplicitybounds}]\ 
	\label{thm:B}\\
	For all $U$-types $V_\lambda$, there is an isomorphism
	\begin{align*}
		\Hom_U(V_\lambda,L^2(X,\sE))\cong\Hom_K(V_\lambda^{\mf n^+},\Poly^2(\Vm,E)).
	\end{align*}
	Moreover, any $K$-type in $\Poly^2(\Vm,E)$ is isomorphic to $V_\lambda^{\mf n^+}$ for some 
	$\lambda$, and the multiplicities of $U$-types in $L^2(X,\sE)$ are uniformly bounded by the 
	dimension of $E$.
\end{theoremAlpha}

This extends our results from \cite{S12a}, where we assumed that $\sE$ admits non-trivial holomorphic sections in which case $\Poly^2(\mf n^-,E)$ coincides with $\Poly(\mf n^-,E)$. Applying this result to the case of a line bundle $\linebundle$, we determine the precise $U$-type decomposition of $L^2(X,\linebundle)$, which recovers Schlichtkrull's generalization of the Cartan--Helgason theorem \cite{Sch84}, see Theorem~\ref{thm:linebundledecomposition}. As an application to higher rank vector bundles, we consider the holomorphic cotangent bundle $T^{(1,0)*}$ and prove some estimates on the multiplicities of $U$-types contained in $L^2(X,T^{(1,0)*})$, see Theorem~\ref{thm:PolyDecomposition}. However, we note that it is quite a hard problem to give precise formulas for the multiplicities. Our estimates are based on the $K$-type decomposition of $\Poly(\mf n^-,E)$, and it remains to determine which $K$-types are actually contained in $\Poly^2(\mf n^-,E)$. So far, we do not have enough insight into the $L^2$-condition \eqref{eq:L2conditionIntro} in order to determine $\Poly^2(\mf n^-,E)$ more explicitly, see Corollary~\ref{cor:propconverse} for a partial result.\\

This paper is organized as follows. Section~\ref{sec:LocalPicture} is devoted to the proof of the first part of Theorem~\ref{thm:A}. In Section~\ref{sec:Existence} we prove the existence of non-trivial nearly holomorphic sections. Section~\ref{sec:Applications} provides the proof of Theorem~\ref{thm:B} (Section~\ref{subsec:mainresult}), discusses some results on the $L^2$-condition (Section~\ref{sec:ResOnL2Cond}), which then yield applications to line bundles (Section~\ref{subsec:LineBundles}) and to the holomorphic cotangent bundle (Section~\ref{subsec:holCotanBundle}). We end up with a conjecture concerning highest weight vectors and the $L^2$-condition, which might be helpful in determining the precise $K$-type decomposition of $\Poly^2(\mf n^-,E)$.

\noindent
\emph{Acknowledgment:} I would like to thank Joachim Hilgert, Bernhard Krötz and Jan Möllers for helpful discussions on the topic of this paper.

\section{The local picture of nearly holomorphic sections}\label{sec:LocalPicture}
Let $(X,h)$ be a Kähler manifold of dimension $n$ with Kähler metric $h$, and let $\sE\to X$ be a holomorphic vector bundle on $X$. Recall the definition of nearly holomorphic sections in $\sE$ given in \cite[\S 1.2]{S12a}. For the purpose of this paper, it suffices to note the following characterization. Suppose we have an open and dense subset $\sU\subseteq X$ with Kähler potential $\Psi$ and coordinate functions $z=(z^1,\ldots, z^n)$. Set
\[
	q_\ell:\sU\to\CC,\ q_\ell(z):=\frac{\partial\Psi}{\partial z^\ell}(z).
\]
For $\b i\in\NN^r$, we use standard multi-index notation, so $|\b i| = i_1+\cdots+i_n$ and $q(z)^\b i = \prod_{\ell=1}^n q_\ell(z)^{i_\ell}$.

\begin{proposition}\label{prop:NHCharacterization}
	A smooth section $f\in C^\infty(X,\sE)$ is nearly holomorphic if and only if its restriction 
	to $\sU$ admits a polynomial expression
	\begin{align}\label{eq:NHdecomposition}
		f(z) = \sum_{|\b i|\in\NN} f_\b i(z)\,q(z)^\b i
	\end{align}
	with holomorphic coefficients $f_\b i\in\sO(\sU,\sE_\sU)$, $f_\b i=0$ for almost all $\b i$.
	In this case, the coefficients $f_\b i$ are unique. 
\end{proposition}
\begin{proof}
By definition, $f$ is nearly holomorphic if and only if it is annihilated by the $m$'th invariant Cauchy Riemann operator $\iCR^m$ for some $m\in\NN$, see \cite[\S\,1]{S12a}. Since $\iCR^m$ is a smooth differential operator, it follows that $f$ is nearly holomorphic if and only if the restriction to $\sU\subseteq X$ is nearly holomorphic. Now, the statement immediately follows from \cite[Proposition~1.5]{S12a}.
\end{proof}

The space of all nearly holomorphic sections on $X$ is denoted by $\sN(X,\sE)$, which by definition is a subspace of $C^\infty(X,\sE)$.

We now turn to Hermitian symmetric spaces. Let $G$ be a simply connected and connected complex simple Lie group with holomorphic involution $\sigma$, let $U\subseteq G$ be a $\sigma$-stable maximal compact subgroup with corresponding Cartan involution $\theta$, and set $X:=U/K$ with $K:=U^\sigma$. Then $L:=G^\sigma$ is a reductive Lie group with 1-dimensional center, and $K\subseteq L$ is a maximal compact subgroup. Let $\mf g$, $\mf u$, $\mf l$ and $\mf k$ denote the Lie algebras of $G$, $U$, $K$ and $L$. There exists a unique element $Z_0$ in the center of $\mf l$ whose adjoint action on $\mf g$ induces the grading $\mf g = \mf n^+\oplus\mf l\oplus\mf n^-$ where $\mf n^\pm$ are the $\pm1$ eigenspaces of $\ad_{Z_0}$. We may identify $X$ with the complex quotient $X\cong G/P$, where $P:=L\exp(\mf n^-)$ is a maximal parabolic subgroup of $G$.

We identify $\mf n^+$ with an open and dense subset of $X$ via the exponential map,
\begin{align}\label{eq:embedding}
	\mf n^+\hookrightarrow X=G/P,\ z\mapsto\exp(z)P,
\end{align}
which is a holomorphic embedding. In the sequel, we refer to this identification by simply writing $\mf n^+\subseteq X$.

Let $\sE=G\times_P E$ be the $G$-homogeneous holomorphic vector bundle on $X$ associated to the holomorphic representation $\rho:P\to\GL(E)$. We assume that $\sE$ is \emph{irreducible} in the sense that $\rho$ is irreducible, which implies that $\rho|_{\exp(\mf n^-)}$ is trivial and $\rho|_L$ is irreducible. Then \eqref{eq:embedding} induces a trivialization $\sE_{\mf n^+}\cong\mf n^+\times E$. In particular, the holomorphic tangent bundle $T^{(1,0)}$ is trivialized to $T^{(1,0)}_{\mf n^+}\cong\mf n^+\times\mf n^+$. If $\Psi$ is a Kähler potential on $\mf n^+$, we may regard $q:=\partial\Psi$ as a map $q:\mf n^+\to(\mf n^+)^*$, and we may further identify $(\mf n^+)^*$ with $\mf n^-$ via (the negative of) the Killing form $\kappa$ of $\mf g$, and obtain
\[
	q:\mf n^+\to\mf n^-\quad\text{with}\quad \partial\Psi(z)(v)=-\kappa(v,q(z))
	\quad\text{for all}\quad v\in\mf n^+.
\]
Applying Proposition~\ref{prop:NHCharacterization} to this setting we find that a smooth section $f\in C^\infty(X,\sE)$ is nearly holomorphic if and only if there exists an $\sO(\mf n^+,E)$-valued polynomial $p_f$ on $\mf n^-$ such that the restriction of $f$ to $\mf n^+\subseteq X$ is given by
\[
	f(z) = p_f(z,q(z)).
\]
Here, $p_f(x,y)$ denotes the evaluation of $p_f\in\Poly(\mf n^-,\sO(\mf n^+,E))$ at $y\in\mf n^-$ and $x\in\mf n^+$. As a consequence of Proposition~\ref{prop:NHCharacterization}, we obtain the embedding
\begin{align}
	\iota_\sN: \sN(X,\sE)\hookrightarrow\Poly(\mf n^-,\sO(\Vp,E)),\
	f\mapsto p_f\,.
\end{align}
The main goal of this section is to characterize the image of $\iota_\sN$ in case of a particularly nice Kähler potential.

\begin{lemma}\label{lem:KinvKPot}
	There exists a unique $K$-invariant Kähler potential of $(X,h)$ on $\mf n^+$ satisfying 
	$\Psi(0)=0$.
\end{lemma}
\begin{proof}
For the existence, see \cite[Lemma~2.6]{S12a} and \eqref{eq:KPot} below. Now assume that $\Psi'$ is a second $K$-invariant Kähler potential on $\mf n^+$. Since $\mf n^+$ is simply connected, $\Psi$ and $\Psi'$ just differ by the real part of a holomorphic function, $f\in\sO(\mf n^+)$, and it suffices to show that $f$ is constant. Invariance under $K$ in particular implies that $\Re f$ is invariant under the torus action $z\mapsto e^{it}z$ for $t\in\RR$. Let $f=\sum_{n\geq0} f_n$ denote the expansion of $f$ into homogeneous polynomials. Since $f_n(e^{it}z) = e^{int}f_n(z)$ and $\Re f_n = (\Re f)_n$ (considered as \emph{real} homogeneous polynomials), we readily obtain
\[
	\Re f_n(z) = \cos(nt)\cdot\Re f_n(z) - \sin(nt)\cdot\Im f_n(z)
\]
for all $t\in\RR$. Setting $t=\tfrac{\pi}{n}$ this yields $\Re f_n(z)=0$ for all $n>0$. Therefore, $\Re f$ is constant, and by holomorphy of $f$ we conclude that $f$ is constant.
\end{proof}

For the following we fix $\Psi$ as the Kähler potential given by Lemma~\ref{lem:KinvKPot}, see \eqref{eq:KPot} for an explicit formula. In order to characterize the image of $\iota_\sN$ we choose a $K$-invariant Hermitian inner product on $E$ (since $E$ is irreducible, this inner product is unique up to a positive constant multiple), which induces a Hermitian structure on $\sE$. Let $L^2(X,\sE)$ denote the corresponding space of $L^2$-sections. Since $\sN(X,\sE)\subseteq L^2 (X,\sE)$ and since $\mf n^+\subseteq X$ is dense, elements $p$ in the image of $\iota_\sN$ necessarily satisfy the \emph{$L^2$-condition}
\begin{align}\label{eq:L2condition}
	p(z,q(z))\in L^2(X,\sE).
\end{align}
The following theorem states that $\iota_\sN$ actually maps $\sN(X,\sE)$ into the subspace $\Poly(\mf n^+\times\mf n^-,E)\subseteq\Poly(\mf n^-,\sO(\mf n^+,E))$ of $E$-valued polynomials on $\mf n^+\times\mf n^-$, and conversely, for a polynomial map $p\in\Poly(\Vp\times\Vm,E)$, the $L^2$-condition is sufficient to prove that $f(z):=p(z,q(z))$ extends to a nearly holomorphic section on $X$. Let $\Poly^2(\Vp\times\Vm,E)$ denote the space of $E$-valued polynomials on $\mf n^+\times\mf n^-$ satisfying the $L^2$-condition \eqref{eq:L2condition}.

\begin{theorem}\label{thm:existenceofnearlyholsec}
	Let $X=G/P$ be a compact Hermitian symmetric space, and $\sE=G\times_P E$ be an irreducible 
	$G$-homogeneous Hermitian vector bundle. The map 
	\[
		\iota_\sP: \Poly^2(\Vp\times\Vm,E)\to L^2(X,\sE),\ p\mapsto f_p(z):=p(z,q(z))
	\]
	is an isomorphism onto $\sN(X,\sE)$ with inverse $\iota_\sN$.
\end{theorem}

The proof of this theorem needs some preparation. It essentially uses the fact that $U$-finite vectors in $L^2(X,\sE)$ are smooth sections, which is an elementary fact in abstract representation theory of compact Lie groups. In order to investigate the $\mf u_\CC$-action and to give an explicit description of the $L^2$-condition \eqref{eq:L2condition}, we use the Jordan theoretic framework for Hermitian symmetric spaces. Recall that the pair $(\mf n^+,\mf n^-)$ is a complex simple Jordan pair with Jordan products
\[
	\mf n^\pm\times\mf n^\mp\times\mf n^\pm\to\mf n^\pm,\
	(x,y,z)\mapsto\JTP{x}{y}{z}:=-[[x,y],z]\,.
\]
In addition, for $x,z\in\mf n^\pm$, $y\in\mf n^\mp$ let $Q_x$, $D_{x,y}$ and $\B{x}{y}$ denote the operators defined by
\[
	Q_xy:=\tfrac{1}{2}\,\JTP{x}{y}{x},\quad
	D_{x,y}z:=\JTP{x}{y}{z},\quad\B{x}{y}:=\Id-D_{x,y} + Q_xQ_y.
\]
The latter operator is called the \emph{Bergman operator}, and its determinant is the power of an irreducible polynomial $\Delta:\mf n^+\times\mf n^-\to\CC$, called the \emph{Jordan pair determinant},
\begin{align}\label{eq:JPDet}
	\Det\B{x}{y}=\Delta(x,y)^p.
\end{align}
Here, $p$ is a structure constant of $X$. Finally, let $z\mapsto \bar z$ denote the Cartan involution of $\mf g$ with respect to $\mf u$ restricted to $\mf n^\pm$. Then
\begin{align}\label{eq:KPot}
	\Psi(z):=2p\log\Delta(z,-\bar z)
\end{align}
is the Kähler potential for $(X,h)$ on $\mf n^+$ described in Lemma~\ref{lem:KinvKPot}, and the corresponding $q$-map is given by
\[
	q(z) =\partial\Psi(z) = \bar z^{-z}
\]
where $\bar z^{-z} = \B{\bar z}{-z}^{-1}(\bar z+Q_{\bar z}z)$, see \cite[Lemma~2.6]{S12a}. For convenience, we set $hz:=\Ad_h(z)$ and $Tz:=\ad_T(z)$ for the adjoint actions of $h\in L$ and $T\in\mf l$ on $z\in\mf n^\pm$. Moreover, we note that $\B{z}{-\bar z}$ and $\B{\bar z}{-z}$ are positive definite operators with respect to the inner products $(v,w):=-\kappa(v,\bar w)$ on $\mf n^\pm$, and there exists an element in $L$ whose adjoint action on $\mf n^+$ and $\mf n^-$ is given by $\B{z}{-\bar z}$ and $\B{\bar z}{-z}^{-1}$. In this way we may consider $\B{z}{-\bar z}$ as an element of $L$, see also \eqref{eq:BergmanDecomp} below. In the following, let $\del$ and $\delbar$ denote the usual holomorphic and anti-holomorphic derivatives. The anti-holomorphic tangent space at $z\in\mf n^+$ is identified with $\mf n^-$.

\begin{proposition}\label{prop:uCactiononlocalnearlyhol}
	For $p\in\Poly^2(\Vp\times\Vm,E)$, the action of $\mf u_\CC=\mf g$ on $f_p(z)=p(z,q(z))$ 
	is given by
	\begin{align*}
		d\pi_\CC(v)f_p(z)
			&= -\del_{(v,0)}p\at{(z,q(z))}\,,\\	
		d\pi_\CC(T)f_p(z)
			&= \big(-\del_{(Tz,\,Tq(z))} + d\rho(T)\big)p\at{(z,q(z))}\,,\\
		d\pi_\CC(w)f_p(z)
			&= \big(-\del_{(Q_zw,\ w-\JTP{q(z)}{z}{v})} + d\rho(D_{z,w})\big)p\at{(z,q(z))}
	\end{align*}
	with $v\in\Vp$, $T\in\mf l$, $w\in\Vm$.
\end{proposition}
This proposition immediately follows from the description of the $u_\CC$-action on arbitrary local sections on $\mf n^+$ given in \cite[Lemma~3.2]{S12a} and the following identities for the $q$-map,
\begin{align*}
	(-\partial_v-\bar\partial_{Q_{\bar z} v})q(z) &= 0\,,\\
	(-\partial_{Tz}-\bar\partial_{T\bar z})q(z) &= -Tq(z)\,,\\
	(-\partial_{Q_zw}-\bar\partial_w)q(z) &= -w + \JTP{q(z)}{z}{w}\,.
\end{align*}
These identities can be proved by standard Jordan theoretic calculations using e.g.\ the list of Jordan identities given in the appendix of \cite{Lo77}. We omit the details.

The second ingredient for the proof of Theorem~\ref{thm:existenceofnearlyholsec} is an explicit formula for the $L^2$-condition.

\begin{proposition}\label{prop:localpicturenorm}
	The $L^2$-norm of $f\in L^2(X,\sE)$ is given by
	\[
		\Norm{f}^2
			= \int_{\mf n^+} \left|\rho(\B{z}{-\bar z}^{-\half})f(z)\right|^2
				\Delta(z,-\bar z)^{-p}d\lambda(z),
	\]
	where $p$ is the structure constant of $X$ given by \eqref{eq:JPDet} and $\lambda$ is the 
	Lebesque measure on $\mf n^+$.
\end{proposition}
\begin{proof}
Recall that continuous sections in $\sE$ can be represented by functions $f:G\to E$ satisfying the equivariance condition $f(gp) = \rho(p)^{-1}f(g)$ for all $g\in G$, $p\in P$. Then the $L^2$-norm is given by $\|f\|^2 =\int_U|f(u)|^2du$ with normalized Haar measure $du$ on $U$. For the transformation to an integral on $\mf n^+$, we note that the Langlands decomposition of the parabolic subgroup $P\subseteq G$ is given by $P=MAN^-$ with $A:=\exp(\mf a)$, $\mf a:=\RR Z_0$, $M:=L_\ss\exp(i\mf a)$ and $N^-:=\exp(\mf n^-)$, where $L_\ss$ denotes the semisimple part of $L$. Let $\alpha_0\in\mf a^*$ be defined by $\alpha_0(Z_0) = 1$. Then due to \cite[V \S\,6]{Kna86},
\[
	\int_U |f(u)|^2du = \int_{\mf n^+} |f(u_z)|^2 e^{-2\delta H(\exp(z))}d\lambda(z)\,,
\]
where $\delta := -n\alpha_0$, and for $z\in\mf n^+$ the elements $u_zK\in U/K$ and $H(\exp(z))\in\mf a$ are determined by the factorization of $\exp(z)\in G$ according to $G = UMAN^-$, i.e., $\exp(z) \in u_z Me^{H(\exp(z))}N^-$. Due to \cite[Lemma~2.5]{S12a},
\[
	u_z = \exp(z)\cdot\B{z}{-\bar z}^{\half}\cdot\exp(\bar z).
\]
Using the relation $\B{\bar z}{-z}^{-\half}\bar z = \bar z^{-z}$, we readily obtain the factorization
\[
	\exp(z) = u_z\big((\Det\B{z}{-\bar z}^{\nicefrac{1}{2n}})\cdot\B{z}{-\bar z}^{-\half}\big)
			\big((\Det\B{z}{-\bar z}^{-\nicefrac{1}{2n}})\cdot\exp(Z_0)\big)\,\exp(-(\bar z^{-z}))\,.
\]
Therefore $H(\exp(z)) = (\log\Det\B{z}{-\bar z}^{-\nicefrac{1}{2n}})\,Z_0$, and it follows that
\[
	e^{2n\,\alpha_0 H(\trans{z})} = (\Det\B{z}{-\bar z})^{-1}
		=\Delta(z,-\bar z)^{-p}\,.
\]
Finally, the equivariance condition on $f$ implies
$f(u_z) = \rho(\B{z}{-\bar z}^{-\half})f(z)$, which yields the assertion.
\end{proof}

\begin{remark}
	If $\sE = \linebundle$ is a line bundle, i.e., if $E$ is 1-dimensional, then $\rho:P\to\GL(E)$ 
	is uniquely determined by the action of the center of $L$, hence by $\nu\in\CC$ (actually 
	$\nu\in\tfrac{n}{p}\,\ZZ$, cf.\ Section~\ref{subsec:LineBundles}) with 
	\[
		\rho(\exp(s\,Z_0)) = e^{\nu\,s}\cdot\Id_E\,.
	\]
	In this case,	$\rho(\B{z}{-\bar z}^\half) = (\det\B{z}{-\bar z})^{\nu/2n}
	= \Delta(z,-\bar z)^{\nu p/2n}$, and the 
	$L^2$-norm simplifies to
	\[
		\Norm{f}^2
			= \int_{\mf n^+} |f(z)|^2\Delta(z,-\bar z)^{\tfrac{p}{n}(\nu-n)}d\lambda(z).
	\]
	For vector bundles of higher rank, the evaluation of $|\rho(\B{z}{-\bar z}^{-\half})f(z)|^2$ is 
	more involved. We turn to this in the next section, however we note in advance that there are 
	estimates
	\begin{align}\label{eq:integrandestimate}
		|f(z)|^2\cdot\Delta(z,-\bar z)^{-\nu_1}
			\leq \left|\rho(\B{z}{-\bar z}^{-\half})f(z)\right|^2
			\leq |f(z)|^2\cdot\Delta(z,-\bar z)^{-\nu_2}
	\end{align}
	for suitable constants $\nu_1,\nu_2\in\ZZ$ just depending on $\rho$, see 
	Proposition~\ref{prop:BergmanEigenvalues} below.
\end{remark}

\begin{proof}[Proof of Theorem~\ref{thm:existenceofnearlyholsec}]
We first show that for any $p\in\Poly^2(\Vp\times\Vm,E)$ the local section $f_p$ indeed extends to a nearly holomorphic section on all of $X$. Due to Proposition~\ref{prop:NHCharacterization}, it suffices to show that $f_p$ extends to a smooth section. We prove that $f_p$ is a $U$-finite vector, which implies smoothness. Let $k=\deg_\Vm p$ be the degree of $p(x,y)$ with respect to $y$, i.e., considered as a polynomial in $\Poly(\Vp,E)[\Vm]$, and let $\sP_k\subseteq\Poly^2(\Vp\times\Vm,E)$ be the subspace of polynomials $g$ with $\deg_\Vm g\leq k$. A close look at Proposition~\ref{prop:uCactiononlocalnearlyhol} shows that the space 
\[
	C_k:=\Set{g(z,q(z))\in C^\infty(\Vp,E)}{g\in\sP_k}
\]
is $\mf u_\CC$-invariant. Since $f_p\in C_k$, it remains to show that $C_k$ is finite dimensional. It suffices to show that the $\mf n^+$-degree of elements in $C_k$ is bounded. For $g\in C_k$ set $d_g:=\deg_\Vp g$. Due to Proposition~\ref{prop:localpicturenorm} the $L^2$-condition on $g$ corresponds to the finiteness of the integral
\[
	I(g):=\int_\Vp \left|\rho(\B{z}{-\bar z}^{-\half})g(z,q(z))\right|^2\,
								 \Delta(z,-\bar z)^{-p} d\lambda(z)<\infty.
\]
Since $q(z) = \bar z^{-z} = \eta(z)/\Delta(z,-\bar z)$ for some $\Vm$-valued real polynomial $\eta\in\Poly_\RR(\Vp,\Vm)$, see \cite[\S7.3f]{Lo77}, it follows that
\[
	g(z,q(z)) = \frac{G(z)}{\Delta(z,-\bar z)^k}
\]
for some $E$-valued real polynomial $G\in\Poly_\RR(\Vp,E)$ of degree $\geq d_g$. Due to \eqref{eq:integrandestimate}, we may estimate
\[
	\left|\rho(\B{z}{-\bar z}^{-\half})g(z,q(z))\right|^2
		\geq |g(z,q(z))|^2\cdot \Delta(z,-\bar z)^{-\nu}
\]
with fixed $\nu\in\ZZ$. The condition $I(g)<\infty$ therefore implies that
\[
	\int_\Vp \frac{|G(z)|^2}{\Delta(z,-\bar z)^{\nu+2k+p}}\, d\lambda(z) < \infty\,.
\]
Since the integrand is a rational function, Fubini's Theorem yields that the degree of $G(z)$ with respect to each variable is smaller than the corresponding degree in $\Delta(z,-\bar z)^{\nu+2k+p}$. This shows that $d_g$ is bounded. Therefore, $C_k$ is finite dimensional, and hence $U.f_p\subseteq C_k$ spans a finite dimensional subspace. We conclude that $f_p$ indeed extends to a nearly holomorphic section in $\sE$.\\
Due to Proposition~\ref{prop:NHCharacterization}, $\iota_\sN$ is a left inverse of $\iota_\sP$. Therefore, $\iota_\sP$ is injective and it remains to prove surjectivity onto $\sN(X,\sE)$. Let $N\subseteq\sN(X,\sE)$ denote the image of $\iota_\sP$. In the case of the trivial line bundle, i.e., for the space $\sN(X)$ of nearly holomorphic functions, surjectivity is proved in \cite[Corollary~2.9]{S12a}. It thus follows that $N$ is an $\sN(X)$-submodule of $\sN(X,\sE)$. Moreover, due to Proposition~\ref{prop:uCactiononlocalnearlyhol}, $N$ is a $\mf u_\CC$-invariant subspace, hence it is also $U$-invariant. We claim that it suffices to show that $N$ is non-trivial. Indeed, in this case it follows from \cite[Lemma~2.3]{S12a} that $N$ is dense in $\sN(X,\sE)$, and since $\sN(X,\sE)$ coincides with the space of $U$-finite vectors in $L^2(X,\sE)$, this implies that $N=\sN(X,\sE)$, cf.\ the proof of Proposition~3.3 in \cite{S12a}. Non-triviality of $N$ is a consequence of Theorem~\ref{thm:MinorIntegral} below, which completes the proof.
\end{proof}

\section{Existence of nearly holomorphic sections}\label{sec:Existence}
The goal of this section is to prove the existence of non-trivial nearly holomorphic sections by showing that $\Poly^2(\mf n^+\times\mf n^-,E)$ is non-trivial. In order to verify the $L^2$-condition \eqref{eq:L2condition} for particular $E$-valued polynomials on $\mf n^+\times\mf n^-$, we first prove some estimates for the action of $\rho(\B{z}{-\bar z})$.

We need some root data. Let $\mf h\subseteq\mf l$ be a $\theta$-stable Cartan subalgebra. Since $\mf l$ and $\mf g$ are of the same rank, $\mf h$ is also a Cartan subalgebra of $\mf g$. Let $\Phi:=\Phi(\mf g,\mf h)$ and $\Phi_c:=\Phi(\mf l,\mf h)$ denote the corresponding root systems, and set $\Phi_{nc}:=\Phi\setminus\Phi_c$. Here, elements in $\Phi_c$ (resp.\ $\alpha\in\Phi_{nc})$ are called \emph{compact} (resp.\ \emph{non-compact}) \emph{roots}. We choose a system  $\Delta=\{\alpha_1,\ldots,\alpha_\ell\}$ of simple roots for $\Phi$, such that $\Delta_c:=\Delta\setminus\{\alpha_1\}$ is a system of simple roots for $\Phi_c$, and such that $\mf n^+$ is the sum of all $\mf g_\alpha$ with $\alpha\in\Phi^+_{nc}$. For each positive root $\alpha\in\Phi^+$, we fix a corresponding $\mf{sl}_2$-triple $(X_\alpha,H_\alpha,Y_\alpha)$ satisfying $\theta(H_\alpha)=-H_\alpha$ and $\theta(X_\alpha) = -Y_\alpha$. Let $\Lambda$ (resp. $\Lambda_c$) denote the set of highest weights parametrizing irreducible finite dimensional representations of $U$ (resp. $K$). Since $U$ and $K$ are compact connected, and $U$ is simply connected,
\begin{align*}
	\Lambda &= \Set{\lambda\in\mf h^*}{\lambda(H_{\alpha})\in\NN\text{ for all $\alpha\in\Delta$}},\\
	\Lambda_c &= \Set{\lambda\in\mf h^*}{\lambda(H_{\alpha_1})\in\ZZ,\ \lambda(H_\alpha)\in\NN\text{ for all $\alpha\in\Delta_c$}}.
\end{align*}
We consider the lexicographic order of roots corresponding to the simple roots $(\alpha_1,\ldots,\alpha_\ell)$, i.e., $\alpha>0$ if $\alpha=\sum m_i\alpha_i$ with $m_i>0$ for the first non-zero coefficient. Let $(\gamma_1,\ldots,\gamma_r)$, be the maximal system of strongly orthogonal roots, such that $\gamma_i$ is the \emph{highest} element of $\Phi_{nc}^+$ strongly orthogonal to $\gamma_j$ for $j<i$, i.e., $\gamma_i\pm\gamma_j$ are no roots. We note that in the context of Jordan theory, the elements $X_{\gamma_i}$ are usually denoted by
\begin{align}\label{eq:tripotents}
	e_i:=X_{\gamma_i}\in\mf n^+\qquad (i = 1,\ldots, r),
\end{align}
and $(e_1,\ldots, e_r)$ is called a frame of tripotents. The \emph{polar decomposition} of $\mf n^+$ with respect to this frame is given by the smooth surjection
\begin{align}\label{eq:polardecomposition}
	K\times[0,\infty)^r\to\mf n^+,\ (k,\b t)\mapsto kz_\b t\quad\text{with}\quad
	z_\b t := t_1e_1+\cdots + t_re_r.
\end{align}
Recall that $kz_\b t$ denotes the adjoint action of $k$ on $z_\b t$. As an application, we note that
\begin{align}\label{eq:BergmanDecomp}
	\B{z_\b t}{-\bar z_\b t} = \prod_{i=1}^r\exp(\ln(1+t_i^2)\,H_{\gamma_i}),
\end{align}
see \cite[\S9.7]{Lo77}. Moreover, the polar decomposition yields the integral formula
\begin{align}\label{eq:intformula}
	\int\limits_{\mf n^+} F(z)d\lambda(z)
		= \int\limits_{[0,\infty)^r}\int\limits_K F(kz_\b t)dk
		\prod_{i=1}^r t_i^{2b+1}\prod_{i<j}|t_i^2-t_j^2|^a dt_1\cdots dt_r\,,
\end{align}
where $dk$ is the (suitably normalized) Haar measure on $K$, and $a,b\in\NN$ are structure constants of $X$, which are related to $p$ and $r$ by
\begin{align}\label{eq:pRelation}
	p=2+a(r-1)+b,
\end{align}
see e.g.\ \cite{FK90}.
\begin{proposition}\label{prop:BergmanEigenvalues}
	For all $z\in\mf n^+$, the operator $\rho(\B{z}{-\bar z})$ is positive definite with respect to 
	the $K$-invariant inner product on $E$. Moreover, 
	\[
		\Delta(z,-\bar z)^{-\mu(H_{\gamma_1})}
			\leq \left|\rho(\B{z}{-\bar z}^{-\half})v\right|^2
			\leq \Delta(z,-\bar z)^{-\mu(H_{\alpha_1})},
	\]
	for all $v\in E$ with $|v|=1$, where $\mu\in\Lambda$ is the highest weight of $(\rho,E)$. If 
	$z=k z_\b t$ as in \eqref{eq:polardecomposition}, then 
	$\Delta(z,-\bar z) = \prod_{i=1}^r (1+t_i^2)$.
\end{proposition}
\begin{proof}
According to the polar decomposition \eqref{eq:polardecomposition}, any element $z\in\mf n^+$ can be written in the form $z = k z_\b t$. Since $\B{z}{-\bar z} = k\B{z_\b t}{-\bar z_\b t} k^{-1}$ it therefore suffices to prove positive definiteness for the case $k=\Id$. Due to \eqref{eq:BergmanDecomp},
\[
	\rho(\B{z_\b t}{-\bar z_\b t}) = \prod_{i=1}^r\exp(\ln(1+t_i^2)\,d\rho(H_{\gamma_i})).
\]
If $E = \bigoplus_{\lambda\in\Phi(E)} E^\lambda$ denotes the decomposition of $E$ into weight spaces, we thus obtain for $v\in E^\lambda$ the relation
\[
	\rho(\B{z_\b t}{-\bar z_\b t})v = \prod_{i=1}^r (1+t_i^2)^{\lambda(H_{\gamma_i})}v.
\]
It follows that $\rho(\B{z_\b t}{-\bar z_\b t})$ is positive definite. For arbitrary $k\in K$, the identity $\Delta(z,-\bar z) = \prod_i(1+t_i^2)$ is proved in \cite[\S\,16.15]{Lo75}, and for the estimates on $|\rho(\B{z}{-\bar z}^{-\half})v|^2$ it now suffices to show that 
\begin{align}\label{eq:weightinequalities}
	\mu(H_{\alpha_1})\leq\lambda(H_{\gamma_i})\leq\mu(H_{\gamma_1})
\end{align}
for all $\lambda\in\Phi(E)$ and $i=1,\ldots,r$. Let $W_c$ denote the Weyl group of $\Phi_c$. Since $\Phi(E)$ is $W_c$-stable, and since for each $i=1,\ldots, r$ there exists an element $w_i\in W_c$ such that $w_i\gamma_1=\gamma_i$, it suffices to prove \eqref{eq:weightinequalities} for $i=1$. Next consider weights of $E$ of the form $\lambda = w.\mu$ with $w\in W_c$. Since $\mu$ is dominant for $\Phi_c^+$, it follows that $\mu-w.\mu$ is positive (see \cite[VI\,\S6]{Bou02}). Therefore, since $\gamma_1$ is dominant we obtain $(\mu-w.\mu)(H_{\gamma_1})\geq 0$, and hence $\mu(H_{\gamma_1})\geq (w.\mu)(H_{\gamma_1})$. We note that $\gamma_1 = w_0\alpha_1$, where $w_0$ is the longest element of the Weyl group $W_c$. The same argument as above applied to $\mu-(w_0w).\mu$ and the anti-dominant element $\alpha_1$ yields $(\mu-(w_0w).\mu)(H_{\alpha_1})\leq 0$ and hence $\mu(H_{\alpha_1})\leq(w.\mu)(H_{\gamma_1})$. We conclude that
\begin{align}\label{eq:weightestimate}
	\mu(H_{\alpha_1})\leq (w.\mu)(H_{\gamma_1})\leq \mu(H_{\gamma_1})\quad\text{for all $w\in W_c$}.
\end{align}
For arbitrary $\lambda\in\Phi(E)$, recall that $\Phi(E)$ is contained in the convex hull of $\Set{w.\mu}{w\in W_c}$, see \cite[VIII\,\S7]{Bou05}. Therefore, $\lambda = \sum_{w\in W_c} c_w\cdot(w.\mu)$ with $0\leq c_w\leq 1$ and $\sum_{w\in W_c} c_w = 1$, and \eqref{eq:weightestimate} implies
\[
	\mu(H_{\alpha_1})\leq\lambda(H_{\gamma_1}) \leq \mu(H_{\gamma_1})
	\quad\text{for all $\lambda\in\Phi(E)$},
\]
This completes the proof.
\end{proof}

Now we have everything in place to prove the existence of nearly holomorphic sections. We consider the following explicitly given polynomials: Let $\NN^r_\geq$ denote the set of all $\b m=(m_1,\ldots, m_r)\in\NN^r$ satisfying $m_1\geq \cdots\geq m_r\geq 0$, and for $\b m\in\NN^r_\geq$ define
\begin{align}\label{eq:JordanMinors}
	p_\b m(y):= \Delta_1(y)^{m_1-m_2}\cdots\Delta_{r-1}(y)^{m_{r-1}-m_r}\cdot\Delta_r(y)^{m_r},
\end{align}
where $\Delta_i(y):=\Delta(\epsilon_i,\bar\epsilon_i - y)$ with $\epsilon_i := e_1+\cdots+e_i$. This is a polynomial on $\mf n^-$. In fact it is well-known that this is a highest weight vector for the induced action of $K$ on $\Poly(\mf n^-)$, see also Section~\ref{sec:ResOnL2Cond}. Here we consider $p_\b m$ as an element of $\Poly(\mf n^+\times\mf n^-)$, constant along $\mf n^+$.

\begin{theorem}\label{thm:MinorIntegral}
	If $\b m\in\NN^r_\geq$ satisfies $m_r+\mu(H_{\alpha_1})\geq 0$, then
	\[
		p_\b m(y)\cdot E\subseteq\Poly^2(\Vp\times\Vm,E).
	\]
	In particular, $\Poly^2(\Vp\times\Vm,E)$ is non-trivial.
\end{theorem}
\begin{proof}
We fix an element $v\in E$ and may assume $|v|^2=1$. Then we have to evaluate the integral
\begin{align*}
	I_{\b m,v}:=\int_\Vp |p_{\b m}(\bar z^{-z})|^2
		\left|\rho(\B{z}{-\bar z}^{-\half})v\right|^2
		\Delta(z,-\bar z)^{-p}d\lambda(z)\,.
\end{align*}
For this, we use the integral formula \eqref{eq:intformula}. For $z=kz_\b t$ as in \eqref{eq:polardecomposition}, Proposition~\ref{prop:BergmanEigenvalues} yields
\[
	|\rho(\B{z}{-\bar z}^{-\half})v|\leq\prod_{i=1}^r(1+t_i^2)^{-\mu(H_{\alpha_1})}.
\]
Moreover, due to \cite[Corollary~2.4]{S13a} the polynomial $p_\b m$ satisfies
\[
	|p_\b m(y)| \leq \prod_{i=1}^r\sigma_i^{m_i}, 
\]
where $\sigma_1\geq\cdots\geq\sigma_r\geq 0$ are the singular values of $y$. Since $\bar z^{-z} = k\sum_{i=1}^r \tfrac{t_i}{1+t_i^2}\,\bar e_i$ (see \cite[\S\,3.18]{Lo77}) and $0\leq\tfrac{t_i}{1+t_i^2}<1$, it follows that
\[
	|p_\b m(\bar z^{-z})|^2 \leq \prod_{i=1}^r \left(\tfrac{t_i}{1+t_i^2}\right)^{2m_r}.
\]
We thus obtain
\[
	I_{\b m,v}\leq c\cdot\int_{[0,\infty)^r}
		\prod_{i=1}^r t_i^{2m_r + 2b+1}(1+t_i^2)^{-2m_r-\mu(H_{\alpha_1})-p}
		\prod_{i<j}|t_i^2-t_j^2|^a dt_1\cdots dt_r,
\]
where $c$ is a positive normalizing constant.
Setting $s_i:=\tfrac{t_i^2}{1+t_i^2}$, this transforms to
\begin{align*}
	I_{\b m,v} \leq \frac{c}{2^r}\cdot \int_{[0,1]^r}
			\prod_{i=1}^r s_i^{m_r+b}(1-s_i)^{m_r+\mu(H_{\alpha_1})}\prod_{i<j}|s_i-s_j|^a ds_1\cdots ds_r\,,
\end{align*}
where we have used the relation $p=2+a(r-1)+b$, see \eqref{eq:pRelation}. Since $a,b,m_r\in\NN$, the assumption $m_r+\mu(H_{\alpha_1})\geq 0$ implies that the integrand of the right hand side is bounded, so the integral is finite (in fact, it is a finite Selberg type integral). This completes the proof.
\end{proof}

As a corollary of Theorem~\ref{thm:MinorIntegral} we obtain in combination with Theorem~\ref{thm:existenceofnearlyholsec} and \cite[Theorem~2.1]{S12a} the following important result.

\begin{corollary}\label{cor:existence}
	The space $\sN(X,\sE)$ of nearly holomorphic sections is dense in $\sC(X,\sE)$ with respect to 
	uniform convergence.
\end{corollary}

\begin{remark}
	In Section~\ref{sec:Applications} we prove the converse implication of
	Theorem~\ref{thm:MinorIntegral}, i.e.,\ $p_\b m(y)\cdot E\subseteq\Poly^2(\Vp\times\Vm,E)$ 
	necessarily implies that $m_r+\mu(H_{\alpha_1})\geq0$, see Corollary~\ref{cor:propconverse}. In 
	the special case $\b m = \b 0$, this yields that $L^2(X,\sE)$ contains smooth sections that 
	trivialize to constant maps on $\Vp\subseteq X$ if and only if $\mu(H_{\alpha_1})\geq 0$, i.e., 
	if and only if the $\Phi^+_c$-dominant weight $\mu$ is also $\Phi^+$-dominant. Since such a 
	section is actually holomorphic, this gives a criterion for the existence of non-trivial 
	holomorphic sections in $\sE$. Conversely, it immediately follows from 
	Proposition~\ref{prop:uCactiononlocalnearlyhol}, that if $\sO(X,\sE)$ is non-trivial, then a 
	corresponding highest weight vector restricts to a constant map on $\Vp\subseteq X$. We thus 
	have recovered the well-known Borel--Weil Theorem for the case of Hermitian symmetric spaces.
\end{remark}

\begin{corollary}[Borel--Weil]\label{cor:BorelWeil}
	Let $\sE=G\times_P E$ be an irreducible homogeneous vector bundle on the Hermitian symmetric 
	space $X=G/P$ of compact type, and let $\mu$ be the highest weight of $E$. Then $\sE$ admits 
	non-trivial holomorphic sections if and only if $\mu$ is $\Phi^+$-dominant, i.e., if and only if
	$\mu(H_{\alpha_1})\geq 0$.
\end{corollary}

\section{Application to harmonic analysis}\label{sec:Applications}

We now turn to the problem of decomposing the space of $L^2$-sections in $\sE$ under the action of $U$. Due to abstract representation theory of compact Lie groups, $L^2(X,\sE)$ decomposes into the Hilbert sum
\begin{align}\label{eq:L2decomposition}
	L^2(X,\sE)= \widehat{\bigoplus_{\lambda\in\Lambda}}\; m^\lambda\cdot V_\lambda\,
\end{align}
of $U$-isotypic components, where $m^\lambda\geq 0$ denotes the multiplicity of the $U$-type $V_\lambda$ of highest weight $\lambda\in\Lambda$. 
Frobenius reciprocity yields that all multiplicities are finite.

\subsection{Main result}\label{subsec:mainresult}
Having proved the existence of non-trivial nearly holomorphic sections (Corollary~\ref{cor:existence}), we obtain the following improvement of \cite[Proposition~3.3]{S12a}.

\begin{proposition}\label{prop:UfiniteAreNearlyHol}
	The space $\sN(X,\sE)$ of nearly holomorphic sections coincides with the space of 
	$U$-finite vectors in $L^2(X,\sE)$.
\end{proposition}

According to Theorem~\ref{thm:existenceofnearlyholsec} we may identify $\sN(X,\sE)$ with the space $\Poly^2(\mf n^+\times\mf n^-,E)$ of $E$-valued polynomials on $\mf n^+\times\mf n^-$ satisfying the $L^2$-condition \eqref{eq:L2condition} via the isomorphism
\[
	\iota_\sP: \Poly^2(\Vp\times\Vm,E)\to \sN(X,\sE),\ p\mapsto f_p(z):=p(z,q(z)).
\]
The corresponding action of $\mf u_\CC$ is given by Proposition~\ref{prop:uCactiononlocalnearlyhol}. Since the Kähler potential used for the definition of $q(z)$ is $K$-invariant (Lemma~\ref{lem:KinvKPot}), it readily follows that the action of $K\subseteq U$ on $p\in\Poly^2(\mf n^+\times\mf n^-,E)$ is given by
\[
	(k.p)(x,y)=\rho(k)p(k^{-1}x,k^{-1}y),
\]
which coincides with the induced action of $K$ on arbitrary $E$-valued polynomials on $\mf n^+\times\mf n^-$. Let
\[
	\Poly^2(\mf n^-,E)\subseteq\Poly^2(\mf n^+\times\mf n^-,E)
\]
denote the $K$-invariant subspace of polynomials that are constant along $\mf n^+$. In order to state our main result recall that if $V_\lambda$ is a $U$-type of highest weight $\lambda\in\Lambda$, then the space $V_\lambda^{\mf n^+}$ of $\mf n^+$-invariants,
\[
	V_\lambda^{\mf n^+} :=\Set{v\in V_\lambda}{Y.v=0\text{ for all }Y\in\mf n^+},
\]
is a $K$-type of highest weight $\lambda$. The following is a generalization of \cite[Theorem~3.5]{S12a}, where we assumed that $\sE$ admits non-trivial holomorphic sections. Recall that $\iota_\sN$ denotes the inverse of $\iota_\sP$.

\begin{theorem}\label{thm:UKequivalence}
	For all $\lambda\in\Lambda$, the map
	\begin{align*}
		\varphi_\lambda:\Hom_U(V_\lambda,L^2(X,\sE))\to\Hom_K(V_\lambda^{\mf n^+},\Poly^2(\Vm,E)),\
		T\mapsto \iota_{\sN}\circ T|_{V_\lambda^{\mf n^+}}
	\end{align*}
	is an isomorphism of vector spaces. Moreover, any $K$-type in $\Poly^2(\Vm,E)$ is isomorphic to $V_\lambda^{\mf n^+}$ for some $\lambda\in\Lambda$, hence there is a bijection between $U$-types in $L^2(X,\sE)$ and $K$-types in $\Poly^2(\mf n^-,E)$.
\end{theorem}
\begin{proof}
This theorem is proved essentially in the same way as \cite[Theorem~3.5]{S12a}. For convenience, we recall the details. Since all $U$-finite vectors are nearly holomorphic (Proposition~\ref{prop:UfiniteAreNearlyHol}), the image of $T\in\Hom_U(V_\lambda,L^2(X,\sE))$ is a subspace of $\sN(X,\sE)$, so $\iota_\sN\circ T$ is well-defined. Moreover, due to Proposition~\ref{prop:uCactiononlocalnearlyhol}, the restriction of $\iota_\sN\circ T$ to $V_\lambda^{\mf n^+}$ maps into $\sP^2(\mf n^-,E)$. Therefore, $\varphi_\lambda$ is well-defined.  It is clear that $\varphi_\lambda$ is linear, and bijectivity follows from abstract, but simple arguments. Indeed, if $\varphi_\lambda(T) = 0$ then $T|_{V_\lambda^{\mf n^+}}=0$, and since $V_\lambda^{\mf n^+}$ is a non-trivial subspace of $V_\lambda$, irreducibility of $V_\lambda$ implies that $\ker T = V_\lambda$, so $T=0$. This proves injectivity of $\varphi_\lambda$. For surjectivity, fix a non-trivial $S\in\Hom_K(V_\lambda^{\mf n^+},\Poly^2(\Vm,E))$ and let $v_\lambda$ be the $U$-highest weight vector in $V_\lambda$. Then, $v_\lambda$ is an element of $V_\lambda^{\mf n^+}$ and it is a $K$-highest vector for $V_\lambda^{\mf n^+}$. Consider $f_\lambda:=\iota_\sP(Sv_\lambda)$. Since $f_\lambda\in\sN(X,\sE)$, it is $U$-finite and hence the action of $U$ on $f_\lambda$ generates a $U$-type $\tilde V_\lambda\subseteq L^2(X,\sE)$ of highest weight $\lambda$. Then, $\varphi_\lambda(T) = S$ for $T$ given by the isomorphism $V_\lambda\cong\tilde V_\lambda$ determined by $Tv_\lambda = f_\lambda$. This proves surjectivity. Moreover if $Sv_\lambda$ is replaced by the highest weight vector of any $K$-type in $\Poly^2(\mf n^-,E)$, the same argument shows that this $K$-type is isomorphic to $V_\lambda^{\mf n^+}$ for some $\lambda\in\Lambda$.
\end{proof}

According to Theorem~\ref{thm:UKequivalence}, the space $\Poly^2(\mf n^-,E)$ decomposes under the action of $K$ into
\[
	\Poly^2(\mf n^-,E) = \bigoplus_{\lambda\in\Lambda} m^\lambda\cdot E_\lambda,
\]
where $E_\lambda$ denotes the $K$-type with highest weight $\lambda$, and $m^\lambda$ is the same multiplicity as in \eqref{eq:L2decomposition}. We may consider $\Poly^2(\mf n^-,E)$ as $K$-invariant subspace of the space $\Poly(\mf n^-,E)$ of all $E$-valued polynomials on $\mf n^-$, whose $K$-type decomposition is denoted by
\begin{align*}
	\Poly(\mf n^-,E) = \bigoplus_{\lambda\in\Lambda_c} M^\lambda\cdot E_\lambda.
\end{align*}
Obviously, $m^\lambda\leq M^\lambda$. Moreover, recall that $\Poly(\mf n^-,E)$ is canonically isomorphic (as $K$-module) to the tensor product $\Poly(\mf n^-)\otimes E$, and the well-known Hua--Kostant--Schmid decomposition of $\Poly(\mf n^-)$ \cite{FK90,Hu63,Sc69} yields
\begin{align}\label{eq:HKSDecomposition}
	\Poly(\mf n^-)\otimes E = \bigoplus_{\b m\in\NN^r_\geq}\Poly_\b m(\mf n^-)\otimes E,
\end{align}
where $\NN_\geq^r$ is defined as in \eqref{eq:JordanMinors} and $\Poly_\b m(\mf n^-)\subseteq \Poly(\mf n^-)$ denotes the $K$-type of highest weight $\gamma_\b m:=m_1\gamma_1+\cdots+m_r\gamma_r$. Set $\Gamma:=\Set{\gamma_\b m}{\b m\in\NN^r_\geq}$, and let $\Lambda_E$ denote the set of highest weights $\lambda\in\Lambda$ with positive multiplicity $m^\lambda>0$. Recall that $\Phi(E)$ denotes the set of weights in $E$.

\begin{corollary}\label{cor:multiplicitybounds}
	Let $\Lambda_E\subseteq\Lambda$ denote the set of highest weights corresponding to $U$-types in 
	$L^2(X,\sE)$ with positive multiplicity $m^\lambda$. There are finitely many 
	$\lambda_1,\ldots,\lambda_s\in\big(\Gamma+\Phi(E)\big)\cap\Lambda$ such that
	\[
		\Lambda_E = \bigcup_{i=1}^s \lambda_i+\Gamma.
	\]
	Moreover, for $\lambda\in\Lambda_E$ the multiplicity $m^\lambda$ is bounded by
	\[
		m^\lambda\leq M^\lambda\leq\dim E,
	\]
	where $M^\lambda$ is the multiplicity of the $K$-type $E_\lambda$ in $\Poly(\mf n^-,E)$.
\end{corollary}
\begin{proof}
It is known that the highest weights $\lambda\in\Lambda_c$ occurring in the tensor product $\Poly_\b m(\mf n^-)\otimes E$ must be of the form $\lambda = \gamma_\b m + \nu$ for some weight $\nu\in\Phi(E)$, and that the multiplicity of $E_\lambda$ is less or equal to the dimension of the weight space $E^\nu\subseteq E$, see e.g.\ \cite{Ku10}. Since the Hua--Kostant--Schmid decomposition of $\Poly(\mf n^-)$ is multiplicity free, this implies that the multiplicities in $\Poly(\mf n^-)\otimes\mf n^-$ cannot exceed $\dim E$. Due to Theorem~\ref{thm:UKequivalence}, the highest weights of $L^2(X,\sE)$ must also be $\Phi^+$-dominant, hence we conclude that
$\Lambda_E$ is contained in $(\Gamma+\Phi(E))\cap\Lambda$. 
Moreover, recall that $\sN(X,\sE)$ is an $\sN(X)$ module \cite[Corollary~1.9]{S12a}, and since the highest weights of $\sN(X)$ are precisely given by $\Gamma$ (see Remark~3.8 in \cite{S12a}), it follows that $\lambda+\Gamma\subseteq\Lambda_E$ for all $\lambda\in\Lambda_E$. Finally, for fixed $\nu\in\Phi(E)$ consider the set $\Lambda_\nu = (\Gamma+\nu)\cap\Lambda_E$. If $\Lambda_\nu$ is non-empty, let $\lambda_\nu\in\Lambda_\nu$ be the minimal element with respect to the order relation $\lambda_\nu\leq\lambda_\nu'$ given by the condition that $\lambda_\nu'-\lambda_\nu\in\Gamma$. Then $\Lambda_\nu = \lambda_\nu+\Gamma$, and since $\Phi(E)$ is finite, this completes the proof.
\end{proof}

\subsection{Results on the $L^2$-condition}\label{sec:ResOnL2Cond}
In order to determine $\Poly^2(\mf n^-,E)$ we have to investigate the $L^2$-condition \eqref{eq:L2condition}, i.e., finiteness of the integral
\begin{align}\label{eq:integralcondition}
	I(p):=\int_{\mf n^+} \left|\rho(\B{z}{-\bar z}^{-\half})p(q(z))\right|^2
		\Delta(z,-\bar z)^{-p}d\lambda(z).
\end{align}
We first note the following necessary condition. Recall the definition of $e_1,\ldots,e_r\in\mf n^+$ in \eqref{eq:tripotents}.

\begin{proposition}\label{prop:necessarycondition}
	If $p\in\Poly^2(\mf n^-,E)$ is a weight vector of weight $\lambda\in\mf h^*$, then
	\begin{align}\label{eq:necessarycondition}
		\deg_{t_i} p(t_1\bar e_1+\cdots+t_r\bar e_r) \leq \lambda(H_{\gamma_i})
	\end{align}
	for all $i=1,\ldots, r$.
\end{proposition}
\begin{proof}
Let $\tau$ denote the representation of $L$ on $\sP^2(\mf n^-,E)$, i.e., $(\tau(h)p)(y) = \rho(h)p(h^{-1}y)$ for $h\in L$. Since $q(z) = \bar z^{-z} = \B{\bar z}{-z}^{-\half}\bar z$ and since $\B{z}{-\bar z}\in L$ acts on $\mf n^-$ by $\B{\bar z}{-z}^{-1}$, it follows that
\[
	\rho(\B{z}{-\bar z}^{-\half})p(q(z)) = \big(\tau(\B{z}{-\bar z}^{-\half})p\big)(z).
\]
Furthermore, using the polar decomposition $z = k z_\b t$ as in \eqref{eq:polardecomposition}, we obtain
\[
	I(p)=\int_{K\times[0,\infty)^r}
			 \left|\big(\tau(\B{z_\b t}{-\bar z_\b t}^{-\half})p_k\big)(\bar z_\b t)\right|^2
			 \omega(\b t)\, dkdt_1\cdots dt_r <\infty,
\]
where 
\[
	p_k:=\tau(k^{-1})p\qquad\text{and}\qquad
	\omega(\b t):=\prod_{i=1}^r\frac{t_i^{2b+1}}{(1+t_i^2)^p}\prod_{i<j}|t_i^2-t_j^2|^a.
\]
By Fubini's theorem the integral over each variable $t_i$ is finite for almost all fixed values of $k$ and $t_j$, $j\neq i$. However, since 
\[
	\tau(\B{z_\b t}{-\bar z_\b t}^{-\half}) = \prod_{i=1}^r\exp(\ln(1+t_i^2)\,d\tau(H_{\gamma_i}))
\]
it follows that the integrand of $I(p)$ is of the form
\[
	\frac{g_k(\b t)}{\prod_i(1+t_i^2)^m}\,\omega(\b t)
\]
for some $m\in\NN$, and $g_k$ is a real polynomial on $\RR^r$ depending continuously on $k\in K$. Therefore, for generic $t_j$, $j\neq i$, the finiteness of the integral over $t_i$ is just a condition on the $t_i$-degree of $g_k$, which depends lower semi-continuously on $k$. Hence, the finiteness of the integral over $t_i$ is independent of the choice of $k$. Choosing $k=\Id$, we obtain $p_{\Id} = p$, and by assumption $\tau(\B{z_\b t}{-\bar z_\b t}^{-\half})p = \prod_{i=1}^r (1+t_i^2)^{-\lambda(H_{\gamma_i})/2}p$. Now the condition
\[
	\int_{[0,\infty)}|p(\bar z_\b t)|^2
	(1+t_i^2)^{-\lambda(H_{\gamma_i})-p}\,t_i^{2b+1}\,\prod_{j\neq i}|t_j^2-t_i^2|^a\,dt_i<\infty
\]
implies
\[
	2\,\deg_{t_i} p  - 2\lambda(H_{\gamma_i})-2p + 2b+1+2a(r-1) < -1,
\]
which is equivalent to $\deg_{t_i} p\leq\lambda(H_{\gamma_i})$, since $p = 2+ a(r-1) + b$ and since $\lambda(H_{\gamma_i})$ is an integer.
\end{proof}

In combination with Theorem~\ref{thm:MinorIntegral}, Proposition~\ref{prop:necessarycondition} yields the following partial result concerning the explicit description of $\Poly^2(\mf n^-,E)$. Recall that the polynomial $p_\b m(y)$ defined in \eqref{eq:JordanMinors} is the highest weight vector of the $K$-type $\Poly_\b m(\mf n^-)$, see \cite{Up86}.

\begin{corollary}\label{cor:propconverse}
	For $\b m\in\NN^r_\geq$, 
	\[
		\Poly_\b m(\mf n^-)\otimes E\subseteq\Poly^2(\mf n^-,E)
	\]
	if and only if $m_r+\mu(H_{\alpha_1})\geq 0$.
\end{corollary}
\begin{proof}
First assume that $m_r+\mu(H_{\alpha_1})\geq 0$. Then Theorem~\ref{thm:MinorIntegral} implies $p_\b m\cdot E\subseteq\Poly^2(\mf n^-,E)$, and since $\Poly^2(\mf n^-,E)$ is $K$-invariant, it follows that $\sP_\b m(\mf n^-)\otimes E\subseteq\Poly^2(\mf n^-,E)$. For the converse, consider the polynomial $p_\b m(y)\cdot v\in\Poly^2(\mf n^-,E)$, where $v\in E$ is a weight vector of $E$ of weight $\nu\in\Phi(E)$. Then $p_\b m(y)\cdot v$ is a weight vector of weight $\gamma_\b m + \nu$, and Proposition~\ref{prop:necessarycondition} implies that
\[
	m_r = \deg_{t_r}(p_\b m\cdot v) \leq (\gamma_\b m + \nu)(H_{\gamma_r})
		= 2m_r + \nu(H_{\gamma_r}).
\]
We may choose $\nu$ such that $\nu(H_{\gamma_r}) = \mu(H_{\alpha_1})$, since $\gamma_r$ and $\alpha_1$ are elements of the same $W_c$-orbit, where $W_c$ is the Weyl group of $\Phi_c$, see also the proof of Proposition~\ref{prop:BergmanEigenvalues}. This completes the proof.
\end{proof}

\begin{remark}
	For rank-1 Hermitian symmetric spaces, i.e., for complex projective spaces $X=\PP^n$, 
	Corollary~\ref{cor:propconverse} determines $\Poly^2(\mf n^-,E)$ inside $\Poly(\mf n^-,E)$ 
	up to finitely many $K$-types.
\end{remark}

So far do not have enough insight into the $L^2$-condition to determine $\Poly^2(\mf n^-,E)$ more precisely. For the case of line bundles, Corollary~\ref{cor:propconverse} provides enough information for a complete description of $\Poly^2(\mf n^-,E)$ (see below). In Section~\ref{subsec:holCotanBundle} we briefly discuss the holomorphic cotangent bundle as an example of a higher rank vector bundle.

\subsection{Application to line bundles}\label{subsec:LineBundles}
Recall that $G$-homogeneous line bundles on $X$ correspond to 1-dimensional representations of $P$, which are parame\-trized by integer multiples of the fundamental weight $\lambda_1\in\Lambda_c$ associated to $\alpha_1$, i.e., $\lambda_1(H_{\alpha_i}) = \delta_{1i}$. For the following, we fix $k\in\ZZ$ and consider the line bundle $\linebundle_k$ associated to the character $\rho_k$ with highest weight $\mu_k:=k\cdot\lambda_1$.

Since $E$ is 1-dimensional the tensor product $\Poly_\b m(\mf n^-)\otimes E$ is an irreducible $K$-module of highest weight $\gamma_\b m + \mu_k$. In this case, Corollary~\ref{cor:propconverse} provides a precise description of the $K$-types satisfying the $L^2$-condition. Applying Theorem~\ref{thm:UKequivalence} to this setting, we therefore obtain the following  characterization of the highest weights in $L^2(X,\linebundle_k)$.

\begin{theorem}\label{thm:linebundledecomposition}
	The decomposition of $L^2(X,\linebundle_k)$ into $U$-types is multiplicity 
	free, and the highest weights occurring in this decomposition are given by
	\[
		\Lambda_k
			:= \set{\lambda_{k,\b m} := \gamma_\b m + \mu_k}
						{\b m\in\NN^r_\geq,\ m_r\geq -k}\,.
	\]
	On $\mf n^+\subseteq X$, the highest weight vector corresponding to
	$\lambda_{k,\b m}\in\Lambda_k$ is given by
	\[
		f_\b m(z) = p_\b m(q(z))
	\]
	with $q(z) = \bar z^{-z}$.
\end{theorem}

\begin{remark}
	This characterization of highest weights in $L^2(X,\linebundle_k)$ was first proved by 
	Schichtkrull \cite{Sch84}. For a comparison, we note that Schlichtkrull uses the reverse order 
	of our system of strongly orthogonal roots, i.e., $\tilde\gamma_i:=\gamma_{r-i+1}$ for 
	$i=1,\ldots, r$, and the parameters 
	\[
		\tilde{\b m}(\lambda)=(\tilde m_1,\ldots,\tilde m_r)\quad\text{with}\quad\tilde m_i
			:= \frac{2\,\langle\lambda|\tilde\gamma_i\rangle}
							 {\langle\tilde\gamma_i|\tilde\gamma_i\rangle}
	\]
	for the description of highest weights in $\Lambda_k$. For $\lambda = \lambda_{k,\b m}$ it 
	easily follows that $\tilde m_i = 2\,m_{r-i+1} + k$. Therefore, 
	Theorem~\ref{thm:linebundledecomposition} states that $\Lambda_k$ consists of those 
	$\lambda\in\mf h^*$ that vanish on $\mf h_+\cap\mf l_\ss\subseteq\mf h$ (where $\mf h_+ = 
	\Set{H\in\mf h}{\gamma_i(H) = 0\text{ for all $i$}}$) and for 
	which $\tilde{\b m}:=\tilde{\b m}(\lambda)$ satisfies $\tilde{\b m}\in\NN^r$ and
	\[
		|k|\leq \tilde m_1\leq\tilde m_2\leq\cdots\leq\tilde m_r,\quad
		(-1)^k = (-1)^{\tilde m_1} = \cdots = (-1)^{\tilde m_r},
	\]
	cf.\ Proposition~7.1 and Theorem~7.2 in \cite{Sch84}.
\end{remark}

\begin{remark}
	We also note that if $X$ is of tube-type 
	(and only in this case), $\lambda_1$ is a linear combination of the strongly orthogonal roots 
	$\gamma_i$. More precisely, $\lambda_1 = \tfrac{1}{2}\sum_i\gamma_i$, and we obtain 
	\[
		\Lambda_k
			= \Set{\gamma_{\b m}}
				{\b m\in \tfrac{1}{2}\NN_\geq^r,\ m_r\geq\tfrac{|k|}{2},\ 
					(-1)^k = (-1)^{2m_1}=\cdots=(-1)^{2m_r}}.
	\]
\end{remark}

\subsection{Application to the holomorphic cotangent bundle}\label{subsec:holCotanBundle}
As an application of our results to higher rank vector bundles, we consider the holomorphic cotangent bundle $T^{(1,0)*}$ of $X$, which corresponds to the irreducible representation of $P$ on $\mf n^-$ given by the adjoint action, so $T^{(1,0)*} = G\times_P\mf n^-$ with $\rho(h) = \Ad_h$ for $h\in P$. As before, we set $hv:=\Ad_hv$ for $h\in L$, $v\in\mf n^-$. The Hermitian structure on $T^{(1,0)*}$ is induced from the $K$-invariant inner product on $\mf n^-$ given by $(v|w):=-\kappa(v,\bar w)$, where $\kappa$ denotes the Killing form of $\mf g$. The highest weight of $\rho$ with respect to $\Phi^+_c$ is $\mu = -\alpha_1$. Since $\mu(H_{\alpha_1}) = -2$, the trivial section is the only holomorphic one, see Corollary~\ref{cor:BorelWeil}, and $\Poly^2(\mf n^-,\mf n^-)$ is a proper subset of $\Poly(\mf n^-,\mf n^-)$, see Corollary~\ref{cor:propconverse}. 

According to Corollary~\ref{cor:multiplicitybounds} it suffices to determine the $K$-type decomposition of $\Poly(\mf n^-,\mf n^-)$ in order to obtain bounds for the multiplicities of the $U$-type decomposition of $L^2(X,\sE)$. Recall the decomposition
\[
	\Poly(\mf n^-,\mf n^-)=\bigoplus_{\b m\in\NN^r_\geq}\Poly_\b m(\mf n^-)\otimes\mf n^-
\]
from \eqref{eq:HKSDecomposition}. Essentially the same arguments as used for the decomposition of $\Poly_\b m(\mf n^-)\otimes\mf n^+$ in \cite[Proposition~3.15]{S12a} yield the following result. Recall that $\Delta_c$ denotes the set of simple compact roots, and that the set of weights of $\mf n^-$ coincides with the set of non-compact negative roots in $\Phi$, denoted by $\Phi_{nc}^-$.

\begin{lemma}\label{lem:tensordecomp}
	For $\b m\in\NN^r_\geq$, the $K$-type decomposition of the tensor product
	$\Poly_\b m(\Vm)\otimes\mf n^-$ is given by
	\[
		\Poly_\b m(\Vm)\otimes\mf n^- = \bigoplus_{\lambda\in\Lambda_\b m(\mf n^-)} E_\lambda
	\]
	where
	\begin{align}\label{eq:CotangentBundleHeighestWeights}
		\Lambda_\b m(\mf n^-)
		:= \left(\gamma_\b m + \Set{\beta\in\Phi_{nc}^-}
			{\begin{aligned}
				&\beta+\alpha\notin\Phi_{nc}^- \text{ for all}\\ 
				&\alpha\in\Delta_c\text{ with }\gamma_\b m(H_{\alpha})=0\end{aligned}}
			\right)\cap\Lambda_c.
	\end{align}
	In particular, $\gamma_{\b m-e_j}\in\Lambda_\b m(\mf n^-)$ if and only if $m_j>m_{j+1}$. Here, 
	$m_{r+1}:=0$.
\end{lemma}

\begin{remark}\label{rmk:simplylaced}
	If the root system $\Phi$ of $\mf g$ is simply laced, then the same argument as in 
	\cite[Remark~3.16]{S12a} shows that the additional condition on $\beta\in\Phi^-_{nc}$ in 
	\eqref{eq:CotangentBundleHeighestWeights} is always satisfied, so 
	\eqref{eq:CotangentBundleHeighestWeights} simplifies to
	\[
		\Lambda_\b m(\mf n^-) = (\gamma_\b m + \Phi_{nc}^-)\cap\Lambda_c.
	\]
\end{remark}

By a straightforward analysis of the weights of $\Poly_\b m(\mf n^-)\otimes\mf n^-$ for various $\b m$ (see \cite[Theorem~3.17]{S12a} for details) we obtain the $K$-type decomposition of $\Poly(\mf n^-,\mf n^-)$ and hence the following estimates for the multiplicities of the $U$-types contained in $L^2(X,T^{(1,0)*})$.

\begin{theorem}\label{thm:PolyDecomposition}
	The $K$-type decomposition of $\Poly(\mf n^-,\mf n^-)$ is given by
	\[
		\Poly(\mf n^-,\mf n^-) = \bigoplus_{\lambda\in\Lambda(\mf n^-)} M^\lambda\cdot E_\lambda,
	\]
	where $\Lambda(\mf n^-) = \bigcup_{\b m\in\NN_\geq^r}\Lambda_\b m(\mf n^-)$ and the multiplicity 
	$M^\lambda$ of $\lambda$ satisfies
	\begin{align*}
		M^\lambda =
			\begin{cases}
				\#\Set{i\in\{1,\ldots,r\}}{m_{i-1}>m_i}
				&,\ \text{if }\lambda = \gamma_\b m\text{ with }\b m\in\NN_\geq^r,\\
				1 &,\ \text{else.}
			\end{cases}
	\end{align*}
	Here, $m_{-1}:=\infty$. Furthermore, the $U$-type $V_\lambda$ occurs in $L^2(X,T^{(1,0)*})$ only 
	if $\lambda\in\Lambda(\mf n^-)$, and its multiplicity $m^\lambda$ is bounded by $M^\lambda$.
\end{theorem}

Due to our general analysis of the $L^2$-condition in Section~\ref{sec:ResOnL2Cond} we can say slightly more. Since $\mu(H_{\alpha_1})=-2$, Corollary~\ref{cor:propconverse} yields
\begin{align}\label{eq:L2ConditionCotangentBundle}
	\Poly_\b m(\mf n^-)\otimes\mf n^-\subseteq\Poly^2(\mf n^-,\mf n^-)\iff m_r\geq 2.
\end{align}
Therefore, $m^\lambda = M^\lambda$ for all $\lambda=\gamma_\b m + \beta\in\Lambda(\mf n^-)$ with $m_r\geq 2$, $\beta\in\Phi^-_{nc}$. However, for $m_r\leq 1$ it is quite a hard problem to decide which $K$-types in $\Poly_\b m(\mf n^-)\otimes\mf n^-$ satisfy the $L^2$-condition. There also might be $K$-types in $\Poly^2(\mf n^-,\mf n^-)$ not contained properly in $\Poly_\b m(\mf n^-)\otimes\mf n^-$ for any $\b m$, so it is not sufficient to consider the $K$-types in $\Poly_\b m(\mf n^-)\otimes\mf n^-$ for each $\b m$ separately. We refer to \cite{S13b} for a more detailed discussion of this in the case where $X$ is a Grassmannian manifold. In that special case we are able to write down explicitly all highest weight vectors of $\Poly(\mf n^-,\mf n^-)$, so the $L^2$-condition becomes as concrete as possible. So far, our investigation indicates that condition \eqref{eq:necessarycondition} of Proposition~\ref{prop:necessarycondition} is not just necessary, but also sufficient -- when applied to highest weight vectors, i.e.,

\begin{conjecture}
	Let $p\in\Poly(\mf n^-,E)$ be a highest weight vector of weight $\lambda\in\mf h^*$. Then
	$p\in\Poly^2(\mf n^-,E)$ if and only if
	\[
		\deg_{t_i} p(t_1\bar e_1+\cdots+t_r\bar e_r) \leq 
		\lambda(H_{\gamma_i})\quad\text{for all $i$.}
	\]
	Moreover, if $p\in\Poly_\b m(\mf n^-)\otimes E$, then $m_i\leq\deg_{t_i} p$.
\end{conjecture}
On order to make use of this conjecture even in cases where highest weight vectors are not known explicitly, it also remains to determine a precise relation between the $t_i$-degrees of $p\in\Poly_\b m(\mf n^-)\otimes E$, its weight $\lambda$, and $\b m$.


\bibliographystyle{amsplain}
\bibliography{bibdb}

\end{document}